\documentclass[10pt]{article}
\usepackage{indentfirst, latexsym, bm, amssymb,amsmath,color}
\usepackage{mathrsfs}

\begin{document}

\newtheorem{theorem}{Theorem}[section]
\newtheorem{corollary}[theorem]{Corollary}
\newtheorem{definition}[theorem]{Definition}
\newtheorem{conjecture}[theorem]{Conjecture}
\newtheorem{question}[theorem]{Question}
\newtheorem{lemma}[theorem]{Lemma}
\newtheorem{proposition}[theorem]{Proposition}
\newtheorem{example}[theorem]{Example}
\newtheorem{problem}[theorem]{Problem}
\newenvironment{proof}{\noindent {\bf
Proof.}}{\rule{3mm}{3mm}\par\medskip}
\newcommand{\remark}{\medskip\par\noindent {\bf Remark.~~}}
\newcommand{\pp}{{\it p.}}
\newcommand{\de}{\em}

\newcommand{\JEC}{{\it Europ. J. Combinatorics},  }
\newcommand{\JCTB}{{\it J. Combin. Theory Ser. B.}, }
\newcommand{\JCT}{{\it J. Combin. Theory}, }
\newcommand{\JGT}{{\it J. Graph Theory}, }
\newcommand{\ComHung}{{\it Combinatorica}, }
\newcommand{\DM}{{\it Discrete Math.}, }
\newcommand{\ARS}{{\it Ars Combin.}, }
\newcommand{\SIAMDM}{{\it SIAM J. Discrete Math.}, }
\newcommand{\SIAMADM}{{\it SIAM J. Algebraic Discrete Methods}, }
\newcommand{\SIAMC}{{\it SIAM J. Comput.}, }
\newcommand{\ConAMS}{{\it Contemp. Math. AMS}, }
\newcommand{\TransAMS}{{\it Trans. Amer. Math. Soc.}, }
\newcommand{\AnDM}{{\it Ann. Discrete Math.}, }
\newcommand{\NBS}{{\it J. Res. Nat. Bur. Standards} {\rm B}, }
\newcommand{\ConNum}{{\it Congr. Numer.}, }
\newcommand{\CJM}{{\it Canad. J. Math.}, }
\newcommand{\JLMS}{{\it J. London Math. Soc.}, }
\newcommand{\PLMS}{{\it Proc. London Math. Soc.}, }
\newcommand{\PAMS}{{\it Proc. Amer. Math. Soc.}, }
\newcommand{\JCMCC}{{\it J. Combin. Math. Combin. Comput.}, }
\newcommand{\GC}{{\it Graphs Combin.}, }

\title{Equitable Partition Theorem of Tensors and  Spectrum of Generalized Power Hypergraphs\thanks{
 This work is supported by   the National Natural Science Foundation of China (No.11601337, 11701372, 11531001),  the Joint NSFC-ISF Research Program (jointly funded by the National Natural Science Foundation of China and the Israel Science Foundation (No. 11561141001)),  the Natural Science Foundation of Shanghai (No.16ZR1422400) and the training program of Shanghai Normal University (No.  SK201602).}}
\author{   Ya-Lei Jin$^1$, Jie Zhang$^2$, Xiao-Dong Zhang$^3$\thanks{ Corresponding author: Xiao-Dong Zhang (Email: xiaodong@sjtu.edu.cn)}\\
{\small $^1$Department of Mathematics,}
{\small Shanghai Normal University }\\
{\small 100 Guilin road, Shanghai 200234, P.~R. China. Email:zzuedujinyalei@163.com}\\
{\small $^2$School of Insurance,}
{\small Shanghai Lixin University of Accounting and Finance}\\
{\small 995 Shangchuan road, Shanghai 201209, P.~R. China. Email: zhangjie.sjtu@163.com}\\
{\small $^3$ School of Mathematical Sciences, MOE-LSC, SHL-MAC,}
{\small Shanghai Jiao Tong University} \\
{\small  800 Dongchuan road, Shanghai, 200240,  P.R. China. Email: xiaodong@sjtu.edu.cn}
}

\date{}
\maketitle
 \begin{abstract}
In this paper, we present an equitable partition theorem of tensors, which gives the relations between $H$-eigenvalues of a tensor and its quotient equitable tensor and extends the equitable  partitions of graphs to hypergraphs.  Furthermore,  with the aid of it, some properties  and $H$-eigenvalues of the generalized power hypergraphs are obtained, which extends some known results, including  some results  of Yuan, Qi and Shao \cite{Yuan2016}.

{{\bf Key words:} The equitable partition; tensor; H-eigenvalue;  generalized power hypergraph; signless Laplacian tensor.
 }

      {{\bf AMS Classifications:} 05C65; 15A69}.
       \end{abstract}
\vskip 0.5cm

\section{Introduction}

An order $k$ and  dimensions $(n_1, \cdots, n_k)$ real {\it tensor} $\mathscr{T}=(t_{i_1,i_2, \cdots, i_k})\in \mathbb{R}^{n_1\times n_2\times \cdots\times n_k}$ is a multidimensional array of order $k$ with $n_1n_2\cdots n_k$ entries, where $i_j\in [n_j]$ for $j=1, \cdots, k$ and $[n]:=\{1, \cdots, n\}$. If $n_1=\cdots=n_k=n$, an order $k$ and dimensions  $(n_1, \cdots, n_k)$ real {\it tensor} is called an order $k$ and  dimension $n$ tensor.  If   $\mathscr{A}=(a_{i_1, \cdots, i_m})\in \mathbb{R}^{n_1\times n_2\times \cdots\times n_2}$ and $\mathscr{B}=(b_{i_2, \cdots, i_{k+1}})\in
\mathbb{R}^{n_2\times n_3\times\cdots\times n_{k+1}}$ are order $m\ge 2$ and $k\ge 1$ tensors, respectively, the product (for example, see \cite{bu20141, Shao2013a}) $ \mathscr{A}\mathscr{B}$ of two tensors $\mathscr{A}$ and $\mathscr{B}$  is tensor $\mathscr{C}$ of order $(m-1)(k-1)+1$
 and dimension $(n_1, n_3, \cdots, n_3, n_4, \cdots, n_4, \cdots, n_{k+1}, \cdots, n_{k+1})$ with  entries
\begin{equation}\label{Eq2}
c_{i\alpha_2\cdots \alpha_{m}}= \sum_{i_2,\cdots,i_m\in [n_2]} a_{ii_2\cdots i_m}b_{i_2\alpha_2}\cdots b_{i_m\alpha_m}, (i\in [n_1],\alpha_2,\cdots, \alpha_m\in [n_3]\times[n_4] \times \cdots \times[n_{k+1}]).
\end{equation}
 Let $\mathscr{T}$ be an order $k$ dimension $n$ real tensor. If there exists a real number $\lambda\in \mathbb{R}$ and a nonzero real vector
 $x=(x_1, \cdots, x_n)^T\in \mathbb{R}^{n}$ such that
 $$\mathscr{T}x=\lambda x^{[k-1]}, $$
 then $\lambda$ is called an $H$-{\it eigenvalue} of $\mathscr{T}$ (see \cite{Lim2005} and \cite{Qi2005}) and
  $x$ is called an {\it eigenvector} of $\mathscr{T}$ corresponding to $H$-eigenvalue $\lambda$, where  $x^{[k-1]}=(x_1^{k-1},\cdots,x_n^{k-1})^T$.
  Moreover, the largest $H$-eigenvalue of $\mathscr{T}$
 is denoted by $\lambda(\mathscr{T})$.

 Let ${\mathscr{H}}=(V(\mathscr{H}), E(\mathscr{H}))$ be a simple (i.e., no loops or multiedges) hypergraph, where the vertex set $V(\mathscr{H})=[n]$ and the edge set $E(\mathscr{H})=\{e_1, \cdots, e_m\}$ with $e_i\subseteq V(\mathscr{H})$ for $i=1, \cdots, m$.  Further, if $|e_i|=k$ for $i=1, \cdots, m$, then $\mathscr{H}$ is called a {\it $k$-uniform hypergraph}.  The {\it degree} of a vertex $v\in V(\mathscr{H})$ in  hypergraph $\mathscr{H}$, written $d_v$,  is the number of edges incident to $v$.
 For any hypergraph $\mathscr{H}$,  there are a few tensors associated with $\mathscr{H}$.
               The {\it adjacency tensor} of a $k$-uniform hypergraph $ \mathscr{H}$ on $n$ vertices is defined as the tensor
                $\mathscr{A}(\mathscr{H})=(a_{i_1\cdots i_k})$ of order $k$ and dimension $n$,
                where
                $$a_{i_1\cdots i_k}=\left\{ \begin{array}{ll}
                \frac{1}{(k-1)!} &\ \ \mbox{if} \ e=\{i_1, \cdots, i_k\}\in E(\mathscr{H}),\\
                0 &\ \ {\mbox{otherwise}.} \end{array}\right.$$
  Moreover, let $ \mathscr{D}(\mathscr{H})$ be an order $k$  and dimension $n$ diagonal tensor whose diagonal entries $d_1, \cdots, d_n$. Then $\mathscr{L}(\mathscr{H})=\mathscr{D}(\mathscr{H})-\mathscr{A}(\mathscr{H})$ and $\mathscr{Q}(\mathscr{H})=\mathscr{D}(\mathscr{H})+\mathscr{A}(\mathscr{H})$ are called the {\it Laplacian} and {\it signless Laplacian tensor} of  $H$, respectively.

   During the past over ten years,  the study of tensors has received  increasing and increasing attention. Hu etc. \cite{Hu2013a}  and Shao  etc. \cite{Shao2013} studied the determinant of tensor $\mathscr{A}$ by using the resultant in \cite{Hu2013a} and their properties.  Chang,  Pearson and Zhang in  \cite{Chang2008} extended  Perron-Frobenius theorem of nonnegative matrices to nonnegative irreducible tensor, Yang and Yang \cite{Yang2010,Yang2011}  further obtained many important results on nonnegative tensors, Friedland, Gaubert and Han \cite{Friedland2013} proved an analog of Perron-Frobenius theorem for polynomial maps with nonnegative coefficients. With the rapid development of spectrum of tensor, spectral hypergraph theory has been more and more interesting, since it is able to disclose some relations between structure properties of hypergraphs and eigenvalues (eigenvector) of tensors associated with it.  Keevash,  Lenz and Mubayi in \cite{kee2014} studied the spectral extremal problem of hypergraphs, which has made much contribution for the extremal hypergraph theory.  Shao, Shan and Wu in \cite{Shao2015} characterized that for a connected $k$-uniform hypergraph $\mathscr{H}$,  $H$-spectra of $\mathscr{L}(\mathscr{H})$ and $\mathscr{Q}(\mathscr{H})$
   are the same if and only if $\mathscr{H}$ is odd-bipartite and $k$ is even.  For more results on the spectral radius of hypergraphs, the readers are referred to \cite{chen, Hu2013, Hu2015,Bu2014,Yue2016}.

   In this paper, we motivated by the study of eigenvalues of a tensor and the relationships between the eigenvalues of hypergraphs and the structure properties. Since it is well known that equitable partitions and divisors represent a powerful tool in spectral graph theory, we extend the equitable partitions of graphs to hypergraphs and exploit regularity properties of a hypergraph to obtain part of the spectrum and give some applications in the spectral hypergraph theory. The rest of this paper is organized as follows. In section 2, some notations and known results are presented. In section 3, we give the equitable partition theorem for tensors. In section 4, we obtain some relations between some properties of the generalized power hypergraphs and eigenvalues of  its signless Laplacian tensor, which extend some known results.

\section{Preliminary}
Before stating the Perron-Frobenius theorem, we introduce the following notation.
 An order $k$ and dimension $n$  tensor $\mathscr{A}=(a_{i_1, \cdots, i_k})$  is associated with an  undirected $k$-partite graph
 $G(\mathscr{A}) =(V, E),$ the vertex set of which is the disjoint union $V =\bigcup_{j=1}^kV_j$ with $V_j=[n]$, $j\in [r]$. The
edge $(i_p, i_q)\in V_p\times V_q,$  $ p
\neq q$ belongs to $E$ if and only if $a_{i_1,i_2,\ldots,i_k} \neq 0$ for some $k-2$ indices
$\{i_1,\ldots, i_k\}\setminus\{i_p, i_q\}$. The tensor $\mathscr{A}$ is called {\it weakly irreducible} if the graph $G(\mathscr{A})$ is  connected.
The most important Perron-Frobenius theorem for nonnegative tensors can be stated as follows.

\begin{theorem}\label{FRL}(\cite{Chang2008,Friedland2013, Yang2010, Yang2011})\label{FRL}
Let $\mathscr{A}$ be an order $k$ and dimension $n$ nonnegative  tensor.  \\
(1) Let $x$ be a positive vector. If  $ax^{[k-1]}\le \mathscr{A}x\le bx^{[k-1]}$, then $a\le \lambda(\mathscr{A})\le b$; if $ax^{[k-1]}<\mathscr{A}x< bx^{[k-1]}$, then $a< \lambda(\mathscr{A})< b$, where $a,b$ are real numbers.\\
(2) If $\mathscr{A}$ is weakly irreducible tensor, then there exists a unique eigenvalue $\lambda(\mathscr{A})$ associated with a positive eigenvector. Moreover, the nonnegative eigenvector is unique up to a multiplicative constant.
\end{theorem}
Further, there are the following results on nonnegative tensors.
\begin{theorem}\label{DS}\cite{Yang2010,Yang2011,Shao2013a}
If the two tensors $\mathscr{A}$ and $\mathscr{B}$ are diagonal similar, i.e., $\mathscr{B}=D^{-(k-1)}\mathscr{A}D$ for some invertible diagonal matrix $D$. Then $\mathscr{A}$ and $\mathscr{B}$  have the same eigenvalues.
\end{theorem}

It is easy to obtain the following result from the Perron-Frobenius theorem.

\begin{corollary}
Let $\mathscr{H}$ be a connected hypergraph. Then  there exists a unique eigenvalue $\lambda(\mathscr{A}(\mathscr{H}))$  and $\lambda(\mathscr{Q}(\mathscr{H}))$  associated with a positive eigenvector, respectively.
\end{corollary}

Moreover, Hu et al. \cite{Hu2013a} presented the definition of  {\it (triangular) block tensors}.
\begin{definition}\cite{Hu2013a}
Let $\mathscr{A}$ be an order $m$ and dimension $n$ tensor. If there exists some integer $k$ with  $1\le k\le n-1$ such that
$$a_{i_1i_2\cdots i_m}=0,~(\mbox{$\forall i_1\in [k]$ and at least one of $\{i_2,\cdots,i_m\}$  not in $[k]$}).$$
Then $\mathscr{A}$ is called a lower triangular block tensor.  Upper triangular block tensors can  be similarly defined.
Furthermore, the lower triangular block tensor  may be written in the following  form:
\begin{displaymath}
\mathscr{A}=\left(\begin{matrix}
\mathscr{A}_1&0&\cdots&0&0\cr
*& \mathscr{A}_2&\cdots&0&0\cr
\vdots&\vdots&\ddots&\vdots&\vdots\cr
*&*&\cdots&\mathscr{A}_{k-1}&0\cr
*&*&\cdots&*&\mathscr{A}_{k}
\end{matrix}\right).
\end{displaymath}
\end{definition}

In addition, Khan and Fan \cite{Khan2015} introduced the idea of the generalized power of a simple graph $G$.
\begin{definition} \cite{Khan2015}
Let $G=(V,E)$ be a simple graph.  For $k\ge 2$ and $1\le s\le \frac{k}{2}$, the {\it generalized power graph} of $G$, denoted by $\mathscr{G}^{k,s}=(V^{k,s},E^{k,s})$, is the $k$-uniform hypergraph with the following vertex set and edge set
$$V^{k,s}=\left(\cup_{v\in V}V_v\right)\cup \left(\cup_{e\in E} V_e\right),~~~~E^{k,s}=\{V_u\cup V_v \cup V_e|e=uv\in E\},$$
where $V_v=\{i_{v,1},\cdots,i_{v,s}\}$ and $ V_e=\{ i_{e,1},\cdots,i_{e,k-2s}\}$ satisfy  $V_u\bigcap V_v=\varnothing$ for $u\neq v$  and $V_v\bigcap V_e=\varnothing$. Moreover, the adjacency, Laplacian and signless Laplacian tensors of $G^{k,s}$ are, for short, denoted by
$\mathscr{A}^{k, s}$, $\mathscr{L}^{k, s}$ and $\mathscr{Q}^{k, s}$, respectively.
\end{definition}
Moreover,  Hu and Qi \cite{hu2014} introduced the idea of odd-bipartite.
\begin{definition}\cite{hu2014}
A $k$-uniform hypergraph $\mathscr{H}=(V(\mathscr{H}), E(\mathscr{H}))$ is called {\it odd-bipartite} hypergraph, if $k$ is even and there exists a partition $V(\mathscr{H})=V_1\bigcup V_2$ such that $|e\bigcap V_1|$ is odd for each $e\in E(\mathscr{H})$.\end{definition}

 {\bf Remark}. 1. If $s=1$, then $\mathscr{G}^{k, s}$ is exactly the $k$-th power hypergraph (see \cite{Hu2013}) of $G$.   In particular, $\mathscr{G}^{2,1}$ is exactly $G$.

 2. If $s<\frac{k}{2}$, then  $\mathscr{G}^{k, s}$ is a cored hypergraph (see \cite{Hu2013}) which is odd-bipartite. 3. If $s=\frac{k}{2}$,  then $\mathscr{G}^{k, \frac{k}{2}}$ is odd-bipartite if and only if  $G$ is bipartite (see \cite{Khan2015}).
   the spectrum of generalized power hypergraphs have been intensively studied (for example, see \cite{Hu2013, Khan2015, Shao2015,Bu2014, Yuan2016}). Recently, Yuan, Qi and Shao \cite{Yuan2016} proved the following results,  which confirms a conjecture of Hu, Qi and Shao \cite{Hu2013}.
\begin{theorem}\cite{Yuan2016}\label{ShaoC}
Let $G=(V,E)$ be a simple graph, $k=2r(>2)$ be even and $G^k=(V^k,E^k)$ be the $k$-power hypergraph of $G$. Let $\mathscr{L}^k$ and $\mathscr{Q}^k$ be the Laplacian and signless Laplacian tensors of $G^k$ respectively. Then $\{\lambda(\mathscr{L}^k)=\lambda(\mathscr{Q}^k)\}$ is a strictly decreasing sequence with respect to $k$ for the maximum degree of $G$ at least 2.
\end{theorem}

\section{Equitable partition theorem}
In order to prove the equitable partition theorem, we first prove the following Lemma.
\begin{lemma}\label{MainT-1}
Let $\mathscr{A}$ and $\mathscr{B}$ be two  order $k$ dimension $n$ and  $m$ tensors, respectively. Then there  exists an $n\times m$ matrix $X$ with rank $r$ such that $\mathscr{A}X=X\mathscr{B}$ if and only if there exist two nonsingular matrices $P_{n\times n} $ and $ Q_{m\times m}$ such that
$P^{-1}\mathscr{A}P$ and $Q^{-1}\mathscr{B}Q$  are upper  and lower triangular block  tensors, respectively, which  can be written in the following form:
\begin{displaymath}
P^{-1}\mathscr{A}P=\left(\begin{matrix}
\mathscr{C}_1&*\cr
0& \mathscr{C}_2
\end{matrix}\right){\mbox{and }}
Q^{-1}\mathscr{B}Q=\left(\begin{matrix}
\mathscr{C}_1&0\cr
*& \mathscr{D}_2
\end{matrix}\right),
\end{displaymath}
where $\mathscr{C}_1$ is an order $k$ dimension $r$ tensor.
\end{lemma}
\begin{proof}
Suppose $\mathscr{A}X=X\mathscr{B}$. If $X$ is a matrix with rank $r$, then there exist nonsingular matrices $P_{n\times n}$ and $ Q_{m\times m}$ such that
\begin{displaymath}
P^{-1}XQ=\left(\begin{matrix}
I_r&0\cr
0& 0
\end{matrix}\right),
\end{displaymath}
where $I_r$ is the $r\times r$ unit matrix. Then
$$\mathscr{A}P\left(\begin{matrix}
I_r&0\cr
0& 0
\end{matrix}\right)Q^{-1}=P\left(\begin{matrix}
I_r&0\cr
0& 0
\end{matrix}\right)Q^{-1}\mathscr{B},$$
which implies that
\begin{equation}\label{EEQ1}
P^{-1}\mathscr{A}P\left(\begin{matrix}
I_r&0\cr
0& 0
\end{matrix}\right)=\left(\begin{matrix}
I_r&0\cr
0& 0
\end{matrix}\right)Q^{-1}\mathscr{B}Q.
\end{equation}
 Let
$$\mathscr{C}=P^{-1}\mathscr{A}P=(c_{i_1\cdots i_n}), \mathscr{D}=Q^{-1}\mathscr{B}Q=(d_{j_1\cdots j_m})\ \mbox{and} \ S=\left(\begin{matrix}
I_r&0\cr
0& 0
\end{matrix}\right)=(s_{ij}).$$
Then $\mathscr{C}S=S\mathscr{D}.$
Furthermore, $$(\mathscr{C}S)_{ii_2\cdots i_n}= \sum_{j_2,\cdots, j_n=1}^nc_{ij_2\cdots j_n}s_{j_2i_2}\cdots s_{j_ni_n}=\left\{ \begin{array}
{l@{\quad \quad}l}
c_{ii_2\cdots i_n},\mbox{if $i_2,\cdots, i_n\in [r]$},\\
0,~~~~~~~\mbox{if one of $i_2,\cdots, i_n$ larger than $r$}.
\end{array}\right.$$
$$(S\mathscr{D})_{ii_2\cdots i_n}= \sum_{j=1}^ns_{ij}d_{ji_2\cdots i_n}=\left\{ \begin{array}
{l@{\quad \quad}l}
d_{ii_2\cdots i_n},\mbox{if $i\in [r]$},\\
0,~~~~~~~\mbox{if $i>r$}.
\end{array}\right.$$
In addition,  $\mathscr{C}$ and $\mathscr{D},$ can be written in the following  form:
\begin{displaymath}
\mathscr{C}=P^{-1}\mathscr{A}P=\left(\begin{matrix}
\mathscr{C}_1&\mathscr{C}_{12}\cr
\mathscr{C}_{21}& \mathscr{C}_2
\end{matrix}\right)\ \mbox{and}\ \
\mathscr{D}=Q^{-1}\mathscr{B}Q=\left(\begin{matrix}
\mathscr{D}_1&\mathscr{D}_{12}\cr
\mathscr{D}_{21}& \mathscr{D}_2
\end{matrix}\right).
\end{displaymath}
Then
\begin{displaymath}
\left(\begin{matrix}
\mathscr{C}_1&0\cr
\mathscr{C}_{21}& 0
\end{matrix}\right)=\mathscr{C}S=
S\mathscr{D}=\left(\begin{matrix}
\mathscr{D}_1&\mathscr{D}_{12}\cr
0& 0
\end{matrix}\right),
\end{displaymath}
which implies that $\mathscr{C}_{21}=0$ and $\mathscr{D}_{12}=0$. Hence $P^{-1}\mathscr{A}P, Q^{-1}\mathscr{B}Q$ have the desired form.

Conversely, let $X=PSQ^{-1}$, where $S=\left(\begin{matrix}
I_r&0\cr
0& 0
\end{matrix}\right)$.  It is easy to see that
$$P^{-1}\mathscr{A}X Q=P^{-1}\mathscr{A}PS=SQ^{-1}\mathscr{B}Q=P^{-1}X\mathscr{B}Q,$$
which implies that $\mathscr{A}X=X\mathscr{B}$. This completes the proof.
\end{proof}
{\bf Remark }: If $k=2$ and $\mathscr{A}X=X\mathscr{B}$ in Lemma~\ref{MainT-1}, then $\mathscr{A}$ and $\mathscr{B}$ have at least $r$ common eigenvalues(counting multiplicities). If $k>2$ and $\mathscr{A}X=X\mathscr{B}$,  it is not known whether  $\mathscr{A}$ and $\mathscr{B}$ have common $H$-eigenvalue. However, if $y$ is an eigenvector of $\mathscr{B}$ corresponding to $\lambda$, by the proof of Lemma~\ref{MainT-1}, it is easy to see that
$$\mathscr{A}(Xy)=(\mathscr{A}X)y=(\mathscr{A}PS)Q^{-1}y=(PSQ^{-1}\mathscr{B}Q)Q^{-1}y=PSQ^{-1}\mathscr{B}y=\lambda Xy^{[k-1]}.$$
If $X$ is chosen such that $Xy^{[k-1]}=(Xy)^{[k-1]}$, then $\lambda$ is also an eigenvalue of tensor $\mathscr{A}$.

Now we are ready to introduce the key idea of equitable partition of tensors in this paper.
Let $\mathscr{A}$ be an order $k$ dimension  $n$ tensor. Let $\{V_1,V_2,\cdots V_m\}$ be a partition of $[n]$, that is, for each $i\in [n]$, there is only one $V_j$ such that $i\in V_j$. Let $\mathscr{A}_{i_1\cdots i_k}$ be a block tensor corresponding to $V_{i_1},\cdots, V_{i_k}$. If  $\mathscr{B}=(b_{i_1i_2\cdots i_k})$ is an order $k$ dimension $m$ tensor with
$$b_{ii_2\cdots i_k}=\frac{1}{|V_i|}\sum_{j\in V_i}~~\sum _{j_2\in V_{i_2},j_3\in V_{i_3},\cdots,j_k\in V_{i_k}}a_{jj_2j_3\cdots j_k},$$
then $\mathscr{B}$ is called {\it quotient tensor} of $\mathscr{A}$ corresponding to $\{V_1,V_2,\cdots V_m\}$.
In addition, a partition  $\{V_1,V_2,\cdots, V_m\}$ of $[n]$ is called an {\it equitable partition} corresponding to tensor $\mathscr{A}$, if
$$\sum _{j_2\in V_{i_2},j_3\in V_{i_3},\cdots,j_k\in V_{i_k}}a_{jj_2j_3\cdots j_k}=b_{ii_2\cdots i_k},$$
for each $V_i$ and $j\in V_i$, $i=1,2,\cdots, m.$  Furthermore,  a quotient tensor $\mathscr{B}$ of $\mathscr{A}$ corresponding to an equitable partition $\{V_1,V_2,\cdots V_m\}$ of $[n]$ is called an {\it equitable quotient tensor} of  $\mathscr{A}$.
Moreover,  the  ${n\times m}$ matrix $X=(x_{ij})$  with $x_{ij}=1$ for $i\in V_j$, and  $0$ otherwise, is called {\it characteristic matrix} corresponding to the partition $\{V_1,V_2,\cdots, V_m\}$ of $[n]$.
Then we have the following result.

\begin{lemma}
Let $\mathscr{A}$ be an order $k$ dimension  $n$ tensor. If $X$ is an characteristic matrix for a partition $\{V_1,V_2,\cdots, V_m\}$ of $[n]$ and
$\mathscr{B}$ is the quotient tensor of $\mathscr{A}$ corresponding to this partition,  then this partition is an equitable partition  if and only if $\mathscr{A}X=X\mathscr{B}$.
\end{lemma}
\begin{proof}
 If $\{V_1,V_2,\cdots, V_m\}$ is an equitable partition of $[n]$ and $X$ is characteristic matrix corresponding to this partition,  then, for $i\in V_{i_1}$,
\begin{eqnarray}
(X\mathscr{B})_{ii_2\cdots i_k}&=&\sum_{j=1}^mx_{ij}b_{ji_2\cdots i_k}=b_{i_1i_2\cdots i_k},\\
(\mathscr{A}X)_{ii_2\cdots i_k}&=&\sum_{j_2,\cdots, j_k=1}^na_{ij_2\cdots j_k}x_{j_2i_2}\cdots x_{j_ki_k}=\sum_{j_2\in V_{i_2},\cdots, j_k\in V_{i_k}}a_{ij_2\cdots j_k}x_{j_2i_2}\cdots x_{j_ki_k} \nonumber\\
&=&\sum_{j_2\in V_{i_2},\cdots, j_k\in V_{i_k}}a_{ij_2\cdots j_k}= b_{i_1i_2\cdots i_k}.
\end{eqnarray}
 Hence $\mathscr{A}X=X\mathscr{B}$. Conversely, by comparing entries of $\mathscr{A}X=X\mathscr{B}$, it is easy to see that  this partition is an equitable partition.
 \end{proof}
 We are ready to present the equitable partition theorem in this paper.
\begin{theorem}\label{MainT2}
Let $\mathscr{A}$ be an order $k$ dimension $n$  real tensor. Let $X$ and $\mathscr{B}$ be  characteristic matrix  and quotient equitable tensor of $\mathscr{A}$ corresponding to an equitable partition  $\{V_1,V_2,\cdots V_m\}$  of $[n]$, respectively.   If $y$ is an eigenvector of $\mathscr{B}$ corresponding to $\lambda$, then $Xy$ is an eigenvector of $\mathscr{A}$ corresponding to $\lambda$.
\end{theorem}
\begin{proof}
By $ \mathscr{A}X=X\mathscr{B}$ and $\mathscr{B}y=\lambda y^{[k-1]}$, we have
$$\mathscr{A}(Xy)=(\mathscr{A}X)y=(X\mathscr{B})y=X(\mathscr{B}y)=X(\lambda y^{[k-1]})=\lambda Xy^{[k-1]}.$$
If $i\in V_{j}$ for $j=1, \ldots, m$, then
\begin{eqnarray*}
(Xy)_i&=&\sum_{t=1}^mx_{it}y_t=y_j,\\
(Xy^{[k-1]})_i&=&\sum_{t=1}^mx_{it}y_t^{k-1}=y_j^{k-1}.
\end{eqnarray*}
Hence $Xy^{[k-1]}=(Xy)^{[k-1]}$ and $Xy$ is a nonzero vector. So $\mathscr{A}(Xy)=\lambda Xy^{[k-1]}=\lambda(Xy)^{[k-1]}$. The proof is completed.
\end{proof}
\begin{corollary}\label{Corollary1}
Let $\mathscr{A}$ be an order $k$ dimension $n$ nonnegative real tensor. Let $X$ and $\mathscr{B}$ be  characteristic matrix  and  weakly irreducible quotient equitable tensor of $\mathscr{A}$ corresponding to an equitable partition  $\{V_1,V_2,\cdots V_m\}$  of $[n]$, respectively.  Then $\lambda(\mathscr{A})=\lambda(\mathscr{B})$.
\end{corollary}
\begin{proof}
By Theorem~\ref{FRL} (2), there exists a positive eigenvector $x$ of $\mathscr{B}$ corresponding to $\lambda(\mathscr{B})$. By theorem~\ref{MainT2}, $Xx$ is a positive eigenvector of $\mathscr{A}$ corresponding to $\lambda(\mathscr{B})$, using Theorem~\ref{FRL} (1), $\lambda(\mathscr{A})=\lambda(\mathscr{B})$.
\end{proof}

\section{The largest $H$-eigenvalue of signless Laplacian tensor of the generalized power hypergraph $G^{k,s}$}
\begin{lemma}\label{MainL1}
Let $G=(V,E)$ be an ordinary  graph, for $k=2r\ge 3$ and $1\le s< \frac{k}{2}$, $G^{k,s}=(V^{k,s},E^{k,s})$ is the generalized power graph of $G$. Let $\mathscr{L}^{k,s}$ and $\mathscr{Q}^{k,s}$ be the Laplacian and signless Laplacian tensors of $G^{k,s}$ respectively. Then $\lambda(\mathscr{L}^{k,s})=\lambda(\mathscr{Q}^{k,s})$.
\end{lemma}
\begin{proof}
Since $s< \frac{k}{2}$, we can suppose that $V_1=\{i_{e,1}|e\in E(G)\}$ and $V_2=V^{k,s}\backslash V_{1}$. Then $|e\cap V_1|=1$ for $e\in E^{k,s}$, by Theorem~2.2 in \cite{Shao2015}, $\mathscr{L}^{k,s}$ and $\mathscr{Q}^{k,s}$ have the same spectrum, so $\lambda(\mathscr{L}^{k,s})=\lambda(\mathscr{Q}^{k,s})$. 
\end{proof}

Let $G^{k,s}=(V^{k,s},E^{k,s})$ be the generalized power graph of $G$, where $k>2,~1\le s \le \frac{k}{2}$ are integers. Let
$$V_e=\{i_{e,1},\cdots,i_{e,k-2s}\},  ~~V_v=\{i_{v,1},\cdots,i_{v,s}\},$$
where $e\in E,~v\in V$ and $V_e$ may be empty set. Let $\mathscr{B}^{k,s}$ be the quotient tensor of $\mathscr{Q}^{k,s}$ corresponding to the partition $\{V_e,V_v|e\in E,v\in V\}$.  If $k>2,~1\le s < \frac{k}{2}$, then  $\mathscr{B}^{k,s}$ is  an order $k$ dimension $(|E|+|V|)$ tensor. Since  the dimension of $\mathscr{B}^{k,s}$ is only dependent on $|E|$ and $|V|$ and  is not dependent on $k$,  the subscripts of the entries of $\mathscr{B}^{k,s}$ can be indexed by the element of $E\cup V$. If $k>2,~s =\frac{k}{2}$, then  $\mathscr{B}^{k,s}$ is  an order $k$ dimension $|V|$ tensor, thus the subscripts of the entries of $\mathscr{B}^{k,s}$ can be indexed by the element of $V$.
\begin{lemma}\label{MainL2}
Let $G^{k,s}=(V^{k,s},E^{k,s})$ be the generalized $k$-power hypergraph of a simple graph $G$ with integer $k>2$. Then $\{V_e,V_v|e\in E,v\in V\}$ is an equitable partition of $V^{k,s}$ corresponding to the partition $\{V_e,V_v|e\in E,v\in V\}$.
\end{lemma}
\begin{proof}
 We assume that
$$j_1 \in V_{i_1},j_2\in V_{i_2},j_3\in V_{i_3},\cdots,j_k\in V_{i_k},$$
by the definition of the edge of $G^{k,s}$ and the definition of $V_e$ and $V_v$ for $ e\in E$ and $v\in V$.  $ j_1j_2\cdots j_k$ is an edge of $G^{k,s}$ if and only if there is an edge $e\in E$, such that
$$V_{i_1'}=\cdots =V_{i_{k-2s}'}=V_e, ~~V_{i_{k-2s+1}'}=\cdots=V_{i_{k-s}'}=V_u,~~V_{i_{k-s+1}'}=\cdots=V_{i_{k}'}=V_v ~~\mbox{and}~~e=uv,$$
where $i_1',i_2',\cdots i_k'$ are  pairwise different elements of $\{i_1,i_2,\cdots, i_k\}$.  We will divide into the following five cases to prove that $\{V_e,V_v|e\in E,v\in V\}$ is an equitable partition of $V^{k,s}$ corresponding to tensor  $\mathscr{Q}^{k,s}$.

{\bf Case 1}. If $V_{i_1}= V_{i_2}=V_{i_3}=\cdots =V_{i_k}=V_e,~e\in E$, for any $j\in V_{i_1}$, we have
$$\sum _{j_2\in V_{i_2},j_3\in V_{i_3},\cdots,j_k\in V_{i_k}}(\mathscr{Q}^{k,s})_{jj_2j_3\cdots j_k}=(\mathscr{Q}^{k,s})_{j\cdots j}=d^{k,s}_{j\cdots j}=1,$$
where $d^{k,s}_{j\cdots j}$ is the diagonal entries of $\mathscr{Q}^{k,s}$ corresponding to vertex $j$.

{\bf Case 2}. If $V_{i_1}= V_{i_2}=V_{i_3}=\cdots =V_{i_k}=V_i$
$,~i\in V$, for $j\in V_{i_1}$, we have
$$\sum _{j_2\in V_{i_2},j_3\in V_{i_3},\cdots,j_k\in V_{i_k}}(\mathscr{Q}^{k,s})_{jj_2j_3\cdots j_k}=(\mathscr{Q}^{k,s})_{i\cdots i}=d^{k,s}_{i\cdots i}
=d_i,$$
where $d^{k,s}_{i\cdots i}$ (and
 $d_i$ )
  is the diagonal element of $\mathscr{Q}^{k,s}$ (degree diagonal matrix $D(G))$corresponding to vertex $i$, respectively.

{\bf Case 3}. If there are $j_1\in V_{i_1}, j_2\in V_{i_2},\cdots ,j_k\in V_{i_k}$ such that $j_1j_2\cdots j_k$ is an edge of $G^{k,s}$ and $V_{i_1}=V_u,$ where $u\in V$, then there is a vertex $v\in V$ such that
$$V_{i_1}=V_{{i_2'}}=\cdots= V_{{i_s'}}=V_u,~~V_{i_{s+1}'}=V_{i_{s+2}'}=\cdots =V_{i_{2s}'}=V_v,~~ V_{i_{2s+1}'}=\cdots = V_{i_{k}'}=V_e,$$
where $uv\in E$  and $i_1,i_2',\cdots i_k'$ are  pairwise different elements of $\{i_1,i_2,\cdots, i_k\}$. For any $j_1\in V_{i_1}$, we have
\begin{eqnarray}\label{L2E1}
\sum _{j_2\in V_{i_2},\cdots,j_k\in V_{i_k}}(\mathscr{Q}^{k,s})_{j_1j_2\cdots j_k}&=&\sum _{j_2\in V_{i_2},\cdots,j_k\in V_{i_k}}(\mathscr{A}^{k,s})_{j_1j_2\cdots j_k}\nonumber\\
&=&\sum _{
\mbox{$
\begin{matrix}
j_2\in V_{i_2},\cdots,j_k\in V_{i_k}\cr
j_1j_2\cdots  j_k\in E^{k,s}
\end{matrix}
$}
}(\mathscr{A}^{k,s})_{j_1j_2\cdots j_k}\nonumber \\
&=&\sum _{
\mbox{$
\begin{matrix}
j_2\in V_{i_2},\cdots,j_k\in V_{i_k}\cr
j_1j_2\cdots  j_k\in E^{k,s}
\end{matrix}
$}
}\frac{1}{(k-1)!}\nonumber \\
&=&\frac{(s-1)!\cdot s!\cdot (k-2s)!}{(k-1)!}.
\end{eqnarray}

{\bf Case 4}. If there are $j_1\in V_{i_1}, j_2\in V_{i_2},\cdots ,j_k\in V_{i_k}$ such that $j_1j_2\cdots j_k$ is an edge of $G^{k,s}$ and $V_{i_1}=V_e,$ where $e\in E$, then there are vertices $u,v\in V$ such that
$$V_{i_1}=V_{i_2'}=\cdots V_{{i_{k-2s}'}}=V_e,~~V_{i_{k-2s+1}'}=\cdots V_{i_{k-s}'}=V_u,~~ V_{i_{k-s+1}'}=\cdots = V_{i_{k}'}=V_v,$$
where $uv\in E$  and $i_1,i_2',\cdots i_k'$ are  pairwise different elements of $\{i_1,i_2,\cdots, i_k\}$. For any $j_1\in V_{i_1}$, we have
\begin{eqnarray}\label{L2E2}
\sum _{j_2\in V_{i_2},\cdots,j_k\in V_{i_k}}(\mathscr{Q}^{k,s})_{j_1j_2\cdots j_k}&=&\sum _{j_2\in V_{i_2},\cdots,j_k\in V_{i_k}}(\mathscr{A}^{k,s})_{j_1j_2\cdots j_k}\nonumber\\
&=&\sum _{\mbox{$
\begin{matrix}
j_2\in V_{i_2},\cdots,j_k\in V_{i_k}\cr
j_1j_2\cdots  j_k\in E^{k,s}
\end{matrix}
$}}(\mathscr{A}^{k,s})_{j_1j_2\cdots j_k}\nonumber\\
&=& \frac{s!\cdot s!\cdot (k-2s-1)!}{(k-1)!}.
\end{eqnarray}

{\bf Case 5}. If $V_{i_1}, V_{i_2},\cdots , V_{i_k}$ are not all equal and there are not $j_1\in V_{i_1}, j_2\in V_{i_2},\cdots ,j_k\in V_{i_k}$ such that $j_1j_2\cdots j_k$ is an edge of $G^{k,s}$. Then, for any $j_1\in V_{i_1}$, we have
\begin{equation}\label{L2E3}
\sum _{j_2\in V_{i_2},\cdots,j_k\in V_{i_k}}(\mathscr{Q}^{k,s})_{j_1j_2\cdots j_k}=\sum _{j_2\in V_{i_2},\cdots,j_k\in V_{i_k}}(\mathscr{A}^{k,s})_{j_1j_2\cdots j_k}=0.
\end{equation}
From the above five cases, we have $\{V_e,V_v|e\in E,v\in V\}$ is an equitable partition of $V^{k,s}$ corresponding to tensor  $\mathscr{Q}^{k,s}$.
\end{proof}

\begin{lemma}\label{weak}
Let $G=(V,E)$ be a connected graph with integers $|V|>2$,  $2<k_1<k_2$ and  $1\le s<\frac{k_1}{2}$. If  $\mathscr{B}^{k_1,s},\mathscr{B}^{k_2,s}$ are the quotient tensors of $\mathscr{Q}^{k_1,s},\mathscr{Q}^{k_2,s}$, respectively,  then $\mathscr{B}^{k_1,s},\mathscr{B}^{k_2,s}$ are weakly irreducible.
\end{lemma}
\begin{proof}
Consider the $G(\mathscr{A})$, we divide the following five cases. 

{\bf Case 1}. If $e_1,e_2\in E$ are different, by equality~(\ref{L2E3}),
$$(\mathscr{B}^{k_2,s})_{e_1i_2\cdots i_{k_2}}=0,\mbox{ for some $i_t=e_2,2\le t\le k_2$,}$$
so $(e_1,e_2)$ is not an edge of $G(\mathscr{A})$.

{\bf Case 2}. If $v_1,v_2\in V$ are different and $v_1v_2\notin E$, by equality~(\ref{L2E3}),
$$(\mathscr{B}^{k_2,s})_{v_1i_2\cdots i_{k_2}}=0,\mbox{ for some $i_t=v_2,2\le t\le k_2$,}$$
so $(v_1,v_2)$ is not an edge of $G(\mathscr{A})$.

{\bf Case 3}. If $v_1,v_2\in V$ are different and $e=v_1v_2\in E$, by equality~(\ref{L2E2}),
$$(\mathscr{B}^{k_2,s})_{v_1v_2e\cdots e}=(\mathscr{B}^{k_2,s})_{v_2v_1e\cdots e}=\frac{s!\cdot s!\cdot (k-2s-1)!}{(k_2-1)!},$$
so $(v_1,v_2),(v_2,v_1)$ are edges of $G(\mathscr{A})$.

{\bf Case 4}. If $v_1\in V,e\in E$ and  $v_1$ is not incident with $e$, by equality~(\ref{L2E3}),
$$(\mathscr{B}^{k_2,s})_{v_1i_2\cdots i_{k_2}}=0,\mbox{ for some $i_t=e,2\le t\le k_2$ and }$$
$$(\mathscr{B}^{k_2,s})_{ei_2\cdots i_{k_2}}=0,\mbox{ for some $i_t=v_1,2\le t\le k_2$,}$$
so $(v_1,e),(e,v_1)$ are not edges of $G(\mathscr{A})$.

{\bf Case 5}. If $v_1\in V,e\in E$ and  $e=v_1v_2,v_2\in V$, by equalities~(\ref{L2E1}) and (\ref{L2E2}),
$$(\mathscr{B}^{k_2,s})_{v_1v_2e\cdots e}=\frac{(s-1)!\cdot s!\cdot (k_2-2s-1)!}{(k_2-1)!} \mbox{ and }(\mathscr{B}^{k_2,s})_{ev_1v_2e\cdots e}=\frac{s!\cdot s!\cdot (k_2-2s-1)!}{(k_2-1)!},$$
so $(e,v_1), (v_1,e)$ are edges of $G(\mathscr{A})$.

Hence $G(\mathscr{A})$ is a strongly connected graph following from the above five cases and $G$ is connected. Thus $\mathscr{B}^{k_2,s}$ is weakly irreducible. Similarly, $\mathscr{B}^{k_1,s}$ is also weakly irreducible.
\end{proof}

\begin{lemma}\label{MainC1}
Let $G=(V,E)$ be a connected graph with integers $|V|>2$,  $2<k_1<k_2$ and  $1\le s<\frac{k_1}{2}$. If  $\mathscr{B}^{k_1,s},\mathscr{B}^{k_2,s}$ are the quotient tensors of $\mathscr{Q}^{k_1,s},\mathscr{Q}^{k_2,s}$, respectively, then  $\lambda(\mathscr{B}^{k_1,s})>\lambda(\mathscr{B}^{k_2,s})>\Delta(G)$, where $\Delta(G)$ is the maximum degree of $G$.
\end{lemma}
\begin{proof} By Lemma~\ref{weak}, $\mathscr{B}^{k_1,s}$ and  $\mathscr{B}^{k_2,s}$ are weakly irreducible.  Hence by Theorem~\ref{FRL}, there is a positive eigenvector $x$ of $\mathscr{B}^{k_2,s}$ corresponding to $\lambda(\mathscr{B}^{k_2,s})$, that is $\mathscr{B}^{k_2,s}x=\lambda(\mathscr{B}^{k_2,s})x^{[k_2-1]}.$ Let $e=v_1v_2\in E$.  Since the $i-$th row sum of $\mathscr{B}^{k,s}$ is equal to the $v-$th row sum of $\mathscr{Q}^{k,s}$ for any $v\in V_i$, we have
\begin{eqnarray}\label{C1E1}
&&\lambda(\mathscr{B}^{k_2,s})x^{k_2-1}_e=\sum_{i_2,\cdots,i_{k_2}}(\mathscr{B}^{k_2,s})_{ei_2\cdots i_{k_2}}x_{i_2}\cdots x_{i_{k_2}}\nonumber \\
&=&x_e^{k_2-1}+{k_2-1 \choose s}\times {k_2-1-s \choose s}\times \frac{s!\cdot s!\cdot (k_2-2s-1)!}{(k_2-1)!}x_e^{k_2-2s-1}x_{v_1}^sx_{v_2}^s\nonumber \\
&=&x_e^{k_2-1}+x_e^{k_2-2s-1}x_{v_1}^sx_{v_2}^s,
\end{eqnarray}
\begin{eqnarray}\label{C1E2}
&&\lambda(\mathscr{B}^{k_2,s})x^{k_2-1}_{v_1}=\sum_{i_2,\cdots,i_{k_2}}(\mathscr{B}^{k_2,s})_{v_1i_2\cdots i_{k_2}}x_{i_2}\cdots x_{i_{k_2}}\nonumber\\
&=&d_{v_1}x_{v_1}^{k_2-1}+\sum_{u\in N_G(v_1)}{k_2-1 \choose s-1}\times {k_2-s \choose s}\times \frac{(s-1)!\cdot s!\cdot (k_2-2s-1)!}{(k_2-1)!}x_{e_u}^{k_2-2s}x_u^sx_{v_1}^{s-1}\nonumber \\
&=&d_{v_1}x_{v_1}^{k_2-1}+\sum_{u\in N_G(v_1)}x_{e_u}^{k_2-2s}x_u^sx_{v_1}^{s-1},
\end{eqnarray}
where $e_u=uv_1$. By Equalities~(\ref{C1E1}), (\ref{C1E2}) and the positive vector $x$, we have $\lambda(\mathscr{B}^{k_2,s})>d_{v_1}\ge 1$ and
\begin{equation}\label{C1E3}
x_e=\left(\frac{x_{v_1}x_{v_2}}{[\lambda(\mathscr{B}^{k_2,s})-1]^\frac{1}{s}}\right)^{\frac{1}{2}}.
\end{equation}
Combining Equalities~(\ref{C1E2}) and (\ref{C1E3}), we have
\begin{eqnarray}\label{C1E4}
\lambda(\mathscr{B}^{k_2,s})x^{k_2-1}_{v_1}&=&\sum_{i_2,\cdots,i_{k_2}}(\mathscr{B}^{k_2,s})_{v_1i_2\cdots i_{k_2}}=d_{v_1}x_{v_1}^{k_2-1}+\sum_{u\in N_G(v_1)}\left(\frac{x_{v_1}x_{u}}{[\lambda(\mathscr{B}^{k_2,s})-1]^\frac{1}{s}}\right)^\frac{k_2-2s}{2}x_u^sx_{v_1}^{s-1}\nonumber \\
&=&d_{v_1}x_{v_1}^{k_2-1}+\sum_{u\in N_G(v_1)}\left(\frac{x_{u}}{[\lambda(\mathscr{B}^{k_2,s})-1]^\frac{1}{s}}\right)^\frac{k_2}{2}(\lambda(\mathscr{B}^{k_2,s})-1)x_{v_1}^\frac{k_2-2}{2}.
\end{eqnarray}
Equality~(\ref{C1E4}) implies that
\begin{eqnarray}\label{C1E5}
\frac{\lambda(\mathscr{B}^{k_2,s})-d_{v_1}}{\lambda(\mathscr{B}^{k_2,s})-1}&=&\sum_{u\in N_G(v_1)}\left(\frac{x_{u}}{[\lambda(\mathscr{B}^{k_2,s})-1]^\frac{1}{s}x_{v_1}}\right)^\frac{k_2}{2}.
\end{eqnarray}
Since $G$ is connected graph and $|V|>2$, then there is a vertex $w\in V$ such that $d_w=\Delta(G)\ge 2$, by equality~(\ref{C1E3}),
$$\lambda(\mathscr{B}^{k_2,s})> d_w=\Delta(G)\ge 2.$$
If $uv_1\in E$, combining equality~(\ref{C1E5}), $\lambda(\mathscr{B}^{k_2,s})-d_{v_1}\le \lambda(\mathscr{B}^{k_2,s})-1$ and $x$ is positive vector, then
\begin{equation}\label{C1E6}0<\frac{x_{u}}{[\lambda(\mathscr{B}^{k_2,s})-1]^\frac{1}{s}x_{v_1}}\le \frac{\lambda(\mathscr{B}^{k_2,s})-d_{v_1}}{\lambda(\mathscr{B}^{k_2,s})-1} \le 1,\end{equation}
the equality of inequality~(\ref{C1E6}) holds only if $d_{v_1}=1$

Let
$$f_{v_1}(k)=\sum_{u\in N_G(v_1)}\left(\frac{x_{u}}{[\lambda(\mathscr{B}^{k_2,s})-1]^\frac{1}{s}x_{v_1}}\right)^\frac{k}{2}.$$
Then $f_{v_1}(k)$ is a decreasing function of $k>2$ independent on $v_1$, since $0<\frac{x_{u}}{[\lambda(\mathscr{B}^{k_2,s})-1]^\frac{1}{s}x_{v_1}}\le 1$ for every $u\in N_G(v_1)$. By $2<k_1<k_2$, we have  $f_{v_1}(k_1)\ge f_{v_1}(k_2)$ and $f_{w}(k_1) > f_{w}(k_2)$. That is
\begin{equation}\label{C1E7}
\sum_{u\in N_G(v_1)}\left(\frac{x_{u}}{[\lambda(\mathscr{B}^{k_2,s})-1]^\frac{1}{s}x_{v_1}}\right)^\frac{k_1}{2}\ge \sum_{u\in N_G(v_1)}\left(\frac{x_{u}}{[\lambda(\mathscr{B}^{k_2,s})-1]^\frac{1}{s}x_{v_1}}\right)^\frac{k_2}{2}
=\frac{\lambda(\mathscr{B}^{k_2,s})-d_{v_1}}{\lambda(\mathscr{B}^{k_2,s})-1},
\end{equation}
\begin{equation}\label{C1E8}
\sum_{u\in N_G(w)}\left(\frac{x_{u}}{[\lambda(\mathscr{B}^{k_2,s})-1]^\frac{1}{s}x_{w}}\right)^\frac{k_1}{2}> \sum_{u\in N_G(w)}\left(\frac{x_{u}}{[\lambda(\mathscr{B}^{k_2,s})-1]^\frac{1}{s}x_{w}}\right)^\frac{k_2}{2}
=\frac{\lambda(\mathscr{B}^{k_2,s})-d_{w}}{(\lambda\mathscr{B}^{k_2,s})-1},
\end{equation}
By $e=v_1v_2$ and equality~(\ref{C1E3}), we have
\begin{eqnarray}\label{C1E9}
(\mathscr{B}^{k_1,s}x)_e&=&\sum_{i_2,\cdots,i_{k_1}}(\mathscr{B}^{k_1,s})_{ei_2\cdots i_{k_1}}x_{i_2}\cdots x_{i_{k_1}}\nonumber \\
&=&x_e^{k_2-1}+{k_2-1 \choose s}\times {k_2-1-s \choose s}\times \frac{s!\cdot s!\cdot (k_1-2s-1)!}{(k_1-1)!}x_e^{k_1-2s-1}x_{v_1}^sx_{v_2}^s\nonumber \\
&=&x_e^{k_1-1}+x_e^{k_1-2s-1}x_{v_1}^sx_{v_2}^s=\lambda(\mathscr{B}^{k_2,s})x^{k_1-1}_e,
\end{eqnarray}
By (\ref{C1E7}) and $x_{v_1}^{k_1-1}>0$, we have
\begin{equation}
\sum_{u\in N_G(v_1)}\left(\frac{x_{u}}{[\lambda(\mathscr{B}^{k_2,s})-1]^\frac{1}{s}x_{v_1}}\right)^\frac{k_1}{2}x_{v_1}^{k_1-1}\ge \frac{\lambda(\mathscr{B}^{k_2,s})-d_{v_1}}{\lambda(\mathscr{B}^{k_2,s})-1}x_{v_1}^{k_1-1}.\nonumber
\end{equation}
Combining with (\ref{C1E7}), we have
\begin{eqnarray}
&&\sum_{u\in N_G(v_1)}\left(\frac{x_{u}}{[\lambda(\mathscr{B}^{k_2,s})-1]^\frac{1}{s}x_{v_1}}\right)^\frac{k_1}{2}x_{v_1}^{k_1-1}=\sum_{u\in N_G(v_1)}\left(\frac{x_{u}^\frac{k_1-2s}{2}x_{v_1}^\frac{k_1-2s}{2}}{[\lambda(\mathscr{B}^{k_2,s})-1]^{\frac{k_1-2s}{2s}}}\right)x_{u}^sx_{v_1}^{s-1}\nonumber \\
&=&\sum_{u\in N_G(v_1)}\left(\frac{x_{u}^\frac{k_1-2s}{2}x_{v_1}^\frac{k_1-2s}{2}}{[\lambda(\mathscr{B}^{k_2,s})-1]^{\frac{k_1-2s}{2s}}}\right)\frac{x_{u}^sx_{v_1}^{s-1}}{\lambda(\mathscr{B}^{k_2,s})-1}\nonumber \\
&=&\sum_{u\in N_G(v_1)}\frac{x_{e_u}^{k_1-2s}x_{u}^sx_{v_1}^{s-1}}{\lambda(\mathscr{B}^{k_2,s})-1}=\frac{(\mathscr{B}^{k_1,s}x)_{v_1}-d_{v_1}x_{v_1}^{k_1-1}}{\lambda(\mathscr{B}^{k_2,s})-1}\nonumber \ge\frac{\lambda(\mathscr{B}^{k_2,s})-d_{v_1}}{\lambda(\mathscr{B}^{k_2,s})-1}x_{v_1}^{k_1-1},\nonumber
\end{eqnarray}
which implies that
\begin{equation}\label{C1E10}
(\mathscr{B}^{k_1,s}x)_{v_1}\ge \lambda(\mathscr{B}^{k_2,s})x_{v_1}^{k_1-1}.
\end{equation}
Similarly, by (\ref{C1E8}), we have
\begin{equation}\label{C1E11}
(\mathscr{B}^{k_1,s}x)_{w}> \lambda(\mathscr{B}^{k_2,s})x_{w}^{k_1-1}.
\end{equation}
By inequalities~(\ref{C1E9}), (\ref{C1E10}), (\ref{C1E11}) and Theorem~\ref{FRL} (1), we have that
$\lambda(\mathscr{B}^{k_1,s})>\lambda(\mathscr{B}^{k_2,s})$.
\end{proof}

\begin{theorem}\label{MainC2}
Let   $\mathscr{B}^{k,s}$ be the quotient tensor of the signless Laplacian tensor $\mathscr{Q}^{k,s}$ of the generalized power graph $G^{k,s}$ associated with a connected graph $G=(V,E)$, where $2<k,2 <|V|$ and $1\le s<\frac{k}{2}$. Then $\lambda(\mathscr{B}^{k,s})=\lambda(\mathscr{Q}^{k,s})$ is strictly decreasing with respect to $k$.
\end{theorem}
\begin{proof}
By Lemma~\ref{MainL2} and Corollary~\ref{Corollary1}, $\lambda(\mathscr{B}^{k,s})=\lambda(\mathscr{Q}^{k,s})$. By Lemma~\ref{MainC1},  $\lambda(\mathscr{B}^{k,s})=\lambda(\mathscr{Q}^{k,s})$ is strictly decreasing with respect to $k$.
\end{proof}
\begin{lemma}\label{lemma}
Let $G=(V,E)$ be a connected graph and $2<k=2r,2<|V|$, $\mathscr{B}^{k,s}$ is the quotient tensor of $\mathscr{Q}^{k,s}$. Then $\lambda(\mathscr{Q}^{k,r})>\lambda(\mathscr{Q}^{k,r-1})$.
\end{lemma}
\begin{proof}
By Lemma~\ref{MainL2} and Corollary~\ref{Corollary1}, $\lambda(\mathscr{B}^{k,r})=\lambda(\mathscr{Q}^{k,r}),~\lambda(\mathscr{B}^{k,r-1})=\lambda(\mathscr{Q}^{k,r-1})$. Next prove $\lambda(\mathscr{B}^{k,r})>\lambda(\mathscr{B}^{k,r-1})$. By Corollary~\ref{MainC1} and Theorem~\ref{FRL}, there is a positive eigenvector $x$ of $\mathscr{B}^{k,r-1}$ corresponding to $\lambda(\mathscr{B}^{k,r-1})$, that is $\mathscr{B}^{k,r-1}x=\lambda(\mathscr{B}^{k,r-1})x^{[k-1]}.$
Let $e=v_1v_2\in E$. Then
\begin{eqnarray}\label{T2E1}
&&\lambda(\mathscr{B}^{k,r-1})x^{k-1}_e=\sum_{i_2,\cdots,i_{k}}(\mathscr{B}^{k,r-1})_{ei_2\cdots i_{k}}x_{i_2}\cdots x_{i_{k}}=x_e^{k-1}+x_ex_{v_1}^{r-1}x_{v_2}^{r-1},
\end{eqnarray}
\begin{eqnarray}\label{T2E2}
&&\lambda(\mathscr{B}^{k,r-1})x^{k-1}_{v_1}=\sum_{i_2,\cdots,i_{k}}(\mathscr{B}^{k,r-1})_{v_1i_2\cdots i_{k}}x_{i_2}\cdots x_{i_{k}}=d_{v_1}x_{v_1}^{k-1}+\sum_{u\in N_G(v_1)}x_{e_u}^2x_u^{r-1}x_{v_1}^{r-2},
\end{eqnarray}
where $e_u=uv_1$. By Equalities~(\ref{T2E1}), (\ref{T2E2}) and the positive vector $x$, we have $\lambda(\mathscr{B}^{k,r-1})>d_{v_1}\ge 1$ and
\begin{equation}\label{T2E3}
x_e=\left(\frac{x_{v_1}x_{v_2}}{[\lambda(\mathscr{B}^{k,r-1})-1]^{\frac{1}{r-1}}}\right)^{\frac{1}{2}}.
\end{equation}
Combining Equalities~(\ref{T2E2}) and (\ref{T2E3}), we have
\begin{eqnarray}\label{T2E4}
&&\lambda(\mathscr{B}^{k,r-1})x^{k-1}_{v_1}=d_{v_1}x_{v_1}^{k-1}+\frac{1}{[\lambda(\mathscr{B}^{k,r-1})-1]^{\frac{1}{r-1}}}\sum_{u\in N_G(v_1)}x_u^{r}x_{v_1}^{r-1}.
\end{eqnarray}
Since $G$ is connected graph and $|V|>2$, then there is a vertex $w\in V$ such that $d_w=\Delta(G)\ge 2$, by Lemma~\ref{MainC1},
$$\lambda(\mathscr{B}^{k_2,r-1})> d_w=\Delta(G)\ge 2.$$
Combining equality~(\ref{T2E4}), we have
\begin{eqnarray}\label{T2E5}
&&\lambda(\mathscr{B}^{k,r-1})x^{k-1}_{v_1}<d_{v_1}x_{v_1}^{k-1}+\sum_{u\in N_G(v_1)}x_u^{r}x_{v_1}^{r-1}.
\end{eqnarray}
Since $G$ is a connected graph, then inequality~(\ref{T2E5}) holds for every vertex of $G$. Let $y$ is the vector such that $y_v=x_v, ~v\in V(G)$. Then for each $v\in V(G)$,
$$(\mathscr{B}^{k,r}y)_v=d_vx_v^{k-1}+\sum_{u\in N_G(v)}x_u^{r}x_{v}^{r-1}>\lambda(\mathscr{B}^{k,r-1})x^{k-1}_{v}.$$
By Theorem~\ref{FRL} (1),
$\lambda(\mathscr{B}^{k,r})>\lambda(\mathscr{B}^{k,r-1})$.
\end{proof}

\begin{theorem}\label{MainC3}
Let $G=(V,E)$ be a connected graph and $2<k,2<|V|$, $1\le s\le \frac{k}{2}$ are integers and $\mathscr{B}^{k,s}$ is the quotient tensor of $\mathscr{Q}^{k,s}$. Then $\lambda(\mathscr{B}^{k,s})=\lambda(\mathscr{Q}^{k,s})$ is strictly increasing  with respect to $s$.
\end{theorem}
\begin{proof}
Suppose $s+1<\frac{k}{2}$, by Lemma~\ref{lemma}, it is sufficient to prove $\lambda(\mathscr{B}^{k,s})<\lambda(\mathscr{B}^{k,s+1})$. By Lemma~\ref{weak} and Theorem~\ref{FRL}, there is a positive eigenvector $x$ of $\mathscr{B}^{k,s}$ corresponding to $\lambda(\mathscr{B}^{k,s})$, that is $\mathscr{B}^{k,s}x=\lambda(\mathscr{B}^{k,s})x^{[k-1]}.$  Let $e=v_1v_2\in E$.  By equalities~(\ref{C1E1}), (\ref{C1E2}) and (\ref{C1E3}), we have
\begin{equation}\label{C3E1}
(\mathscr{B}^{k,s+1}x)_e>\lambda(\mathscr{B}^{k,s})x^{k-1}_e.
\end{equation}
by equalities~(\ref{C1E2}), (\ref{C1E4}), (\ref{C1E5}) and (\ref{C1E6}), we have
\begin{eqnarray}\label{C3E2}
\frac{\lambda(\mathscr{B}^{k,s})-d_{v_1}}{\lambda(\mathscr{B}^{k,s})-1}&=&\sum_{u\in N_G(v_1)}\left(\frac{x_{u}}{[\lambda(\mathscr{B}^{k,s})-1]^{\frac{1}{s}}x_{v_1}}\right)^\frac{k}{2}\\
&<&\sum_{u\in N_G(v_1)}\left(\frac{x_{u}}{[\lambda(\mathscr{B}^{k,s})-1]^{\frac{1}{s+1}}x_{v_1}}\right)^\frac{k}{2}.
\end{eqnarray}
By $x_{v_1}^{k-1}>0$, we have
\begin{equation}
\sum_{u\in N_G(v_1)}\left(\frac{x_{u}}{[\lambda(\mathscr{B}^{k,s})-1]^{\frac{1}{s+1}}x_{v_1}}\right)^\frac{k}{2}x_{v_1}^{k-1}> \frac{\lambda(\mathscr{B}^{k,s})-d_{v_1}}{\lambda(\mathscr{B}^{k,s})-1}x_{v_1}^{k-1}.\nonumber
\end{equation}
Combining with equality~(\ref{C1E7}), we have
\begin{eqnarray}
&&\sum_{u\in N_G(v_1)}\left(\frac{x_{u}}{[\lambda(\mathscr{B}^{k,s})-1]^{\frac{1}{s+1}}x_{v_1}}\right)^\frac{k}{2}x_{v_1}^{k-1}=\sum_{u\in N_G(v_1)}\left(\frac{x_{u}^\frac{k-2s-2}{2}x_{v_1}^\frac{k-2s-2}{2}}{(\lambda(\mathscr{B}^{k,s})-1)^{\frac{k}{2s+2}}}\right)x_{u}^{s+1}x_{v_1}^{s}\nonumber \\
&=&\sum_{u\in N_G(v_1)}\left(\frac{x_{u}^\frac{k-2s-2}{2}x_{v_1}^\frac{k-2s-2}{2}}{(\lambda(\mathscr{B}^{k,s})-1)^{\frac{k-2s-2}{2s+2}}}\right)
\frac{x_{u}^{s+1}x_{v_1}^{s}}{\lambda(\mathscr{B}^{k,s})-1}\nonumber \\
&>&\sum_{u\in N_G(v_1)}\frac{x_{e_u}^{k-2s-2}x_{u}^{s+1}x_{v_1}^{s}}{\lambda(\mathscr{B}^{k,s})-1}
=\frac{(\mathscr{B}^{k,s+1}x)_{v_1}-d_{v_1}x_{v_1}^{k-1}}{\lambda(\mathscr{B}^{k,s})-1}\nonumber > \frac{\lambda(\mathscr{B}^{k,s})-d_{v_1}}{\lambda(\mathscr{B}^{k,s})-1}x_{v_1}^{k-1},\nonumber
\end{eqnarray}
which implies that
\begin{equation}\label{C3E3}
(\mathscr{B}^{k,s+1}x)_{v_1}> \lambda(\mathscr{B}^{k,s})x_{v_1}^{k-1}.
\end{equation}
By inequalities~(\ref{C3E1}), (\ref{C3E3}) and Theorem~\ref{FRL} (1), we have $\lambda(\mathscr{B}^{k,s+1})>\lambda(\mathscr{B}^{k,s})$.
\end{proof}

\begin{corollary}
Let $G=(V,E)$ be a simple  graph with at least one edge,  where $k=2r(>2)$ is even integer, $1\le s<\frac{k}{2}$.\\
(1) If the maximal degree $\Delta(G)$ of $G$ is $1$, then $\{\lambda(\mathscr{L}^{k,s})=\lambda(\mathscr{Q}^{k,s})=\lambda(\mathscr{L}^{k+2})=\lambda(\mathscr{Q}^{k+2})\}$.\\
(2) If the maximal degree $\Delta(G)$ of $G$ is more than $1$, then $\{\lambda(\mathscr{L}^{k,s})=\lambda(\mathscr{Q}^{k,s})\}$ is a strictly decreasing sequence with respect to $k$.
\end{corollary}
\begin{proof}
By theorem~$4.2$ of \cite{Hu2013a}, it is sufficient to consider that $G$ is connected graph. Next we assume $G$ is connected.\\
(1) $\Delta(G)=1$, then $G$ have only two vertices, then $\lambda(\mathscr{L}^{k,s})=\lambda(\mathscr{Q}^{k,s})=\lambda(\mathscr{L}^{k+2})=\lambda(\mathscr{Q}^{k+2})=2$.\\
(2) By Theorems~\ref{MainC1} and \ref{MainC2}, $\{\lambda(\mathscr{L}^{k,s})=\lambda(\mathscr{Q}^{k,s})\}$ is a strictly decreasing sequence.
\end{proof}

\begin{corollary}\cite{Yuan2016}
Let $G=(V,E)$ be a simple  graph with at least one edge and   even integer $k=2r(>2)$. Let $ \mathscr{L}^{k}$ and $\mathscr{Q}^{k}$ are  (signless) Laplacian tensors of the $k$-power hypergraph $G^k$.\\
(1) If the maximal degree $\Delta(G)$ of $G$ is $1$, then $\{\lambda(\mathscr{L}^{k})=\lambda(\mathscr{Q}^{k})=\lambda(\mathscr{L}^{k+2})=\lambda(\mathscr{Q}^{k+2})\}$.\\
(2) If the maximal degree $\Delta(G)$ of $G$ is more than $1$, then $\{\lambda(\mathscr{L}^{k})=\lambda(\mathscr{Q}^{k})\}$ is a strictly decreasing sequence with respect to $k$.
\end{corollary}
\begin{lemma}\label{limlemma}
Let $G=(V,E)$ be a connected $d$-regular graph, where $d>1$, $k>2,~1\le s\le \lfloor\frac{k-1}{2}\rfloor$ are integers. Then $\lambda(\mathscr{Q}^{k,s})$ is the largest root of $(x-d)(x-1)^{\frac{k-2s}{2s}}-d=0.$
\end{lemma}
\begin{proof}
Since $1\le s\le \lfloor\frac{k-1}{2}\rfloor$ implies $s<\frac{k}{2}$, combining Lemma~\ref{MainC1} and Theorem~\ref{MainC2}, we have $\lambda(\mathscr{B}^{k,s})=\lambda(\mathscr{Q}^{k,s})>2$. It is sufficient to prove $\lambda(\mathscr{B}^{k,s})$ is the largest root of $(x-d)(x-1)^{\frac{k-2s}{2}}-d=0.$ Let $y$ be a vector of dimension $|E|+|V|$ and
$$y_i=\left\{
   \begin{array}{c}
  1,~~~~~~~~~~~~\mbox{if $i\in V$}; \\
 \frac{1}{(\lambda -1)^{\frac{1}{2s}}},~~~\mbox{if $i\in E$}, \\
   \end{array}
\right.$$
where $\lambda$ is the largest root of $(x-d)(x-1)^{\frac{k-2s}{2s}}-d=0.$  For $e=vw$,
$$(\mathscr{B}^{k,s}y)_e=y_e^{k-1}+y_e^{k-2s-1}y_u^sy_v^s
=\lambda \left(\frac{1}{\lambda -1}\right)^{\frac{k-1}{2s}}=\lambda y_e^{k-1},$$
and
$$(\mathscr{B}^{k,s}y)_v=d_v+\sum_{u\in N_G(v)}y_{e_u}^{k-2s}y_u^sy_v^{s-1}=d+\frac{d}{(\lambda -1)^{\frac{k-2s}{2s}}}=\lambda=\lambda y_v.$$
By Theorem~\ref{FRL}, $\lambda=\lambda(\mathscr{B}^{k,s})$. The proof is completed.
\end{proof}
\begin{lemma}\cite{konig1936}\label{konig}
Every graph $G$ of maximum degree $\Delta$ is an induced subgraph of some $\Delta$-regular graph.
\end{lemma}
\begin{theorem}
Let $G=(V,E)$ be a graph of maximum degree $\Delta>1$, where $k>2,~1\le s\le \lfloor\frac{k-1}{2}\rfloor$ are integers. Then $\lim_{k\rightarrow \infty}\lambda(\mathscr{Q}^{k,s})=\Delta.$
\end{theorem}
\begin{proof}
By Lemma~\ref{MainC1}, $\lambda(\mathscr{Q}^{k,s})>\Delta $,  it is sufficient to prove $\lim_{k\rightarrow \infty}\lambda(\mathscr{Q}^{k,s})\le\Delta.$ By theorem~$4.2$ of \cite{Hu2013a}, it is sufficient to consider that $G$ is connected graph. By Proposition 4.5 in \cite{Hu2015}, if $H'$ is a sub-hypergraph of $H$, then $\lambda(Q(H'))\le \lambda(Q(H))$. By Lemmas~\ref{konig} and \ref{limlemma},  $\lambda(\mathscr{Q}^{k,s})$ is not more than the largest root of $(x-\Delta)(x-1)^{\frac{k-2s}{2s}}-\Delta=0.$  Since
$$x-\Delta=\lim_{k\rightarrow\infty} \frac{\Delta}{(x-1)^{\frac{k-2s}{2s}}}=0,\mbox{~~for $x>2$},$$
$\lim_{k\rightarrow \infty}\lambda(\mathscr{Q}^{k,s})\le\Delta$. The proof is completed.
\end{proof}

\textbf{Acknowledgments}
The authors would like to thank  the referees for their valuable corrections and suggestions
which lead to a great improvement of this paper.


\begin{thebibliography}{1}

\bibitem{bu20141}C.~Bu, X.~Zhang, J.~Zhou, W.~Wang, Y.~Wei,   The inverse, rank and product of tensors, {\it  Linear Algebra Appl.,} 446 (2014) 269-280.

    \bibitem{chen}D.-M.~Chen, Z.-B.~Chen, X.-D.~Zhang,  Spectral radius of uniform hypergraphs and degree sequences, {\it Front. Math. China,} 12(6) (2017) 1279-1288.
\bibitem{Chang2008}K.-C.~Chang, K.~Pearson, T.~Zhang, Perron-Frobenius theorem for nonnegative tensors, {\it Commun. Math. Sci.,} 6 (2008) 507-520.
\bibitem{Friedland2013}S.~Friedland, S.~Gaubert, L.~Han, Perron-Frobenius theorem for nonnegative multilinear forms and extensions, {\it Linear Algebra Appl.,} 438 (2013) 738-749.

    \bibitem{hu2014} S.~Hu, L.~Qi, The eigenvectors associated with the zero eigenvalues of the Laplacian and signless Laplacian tensors of a uniform hypergraph, {\it  Discrete Appl. Math.,} 169(2014) 140-151.
 \bibitem{Hu2013a}S.~Hu, Z.~Huang, C. Ling, L.~Qi, On determinants and eigenvalue theory of tensors, {\it J. Symbolic Comput.,} 50 (2013)  508-531.
\bibitem{Hu2013}S.~Hu, L.~Qi, J.~Shao, Cored hypergraphs, power hypergraphs and their Laplacian H-eigenvalues. {\it Linear Algebra Appl.,} 439 (2013) 2980-2998.
\bibitem{Hu2015}S.~Hu, L.~Qi, J.~Xie, The largest Laplacian and signless Laplacian H-eigenvalues of a uniform hypergraph, {\it Linear Algebra Appl.,} 469 (2015) 1-27.
\bibitem{Khan2015}M.~Khan, Y.~Fan, On the spectral radius of a class of non-odd-bipartite even uniform hypergraphs, {\it Linear Algebra Appl.,} 480 (2015) 93-106.


\bibitem{kee2014}P.~Keevash, J.~Lenz, D.~Mubayi, Spectral extremal problems for hypergraphs, {\it SIAM J. Discrete Math.,} 28(4)(2014) 1838-1854.

\bibitem{konig1936}D. K\"{o}nig, Theorie der endlichen und unendlichen Graphen, Akademische Verlagsgesellschaft, {\ it Leipzig,} 1936.

\bibitem{Lim2005}L.~Lim. Singular values and eigenvalues of tensors: a variational approach, {\it In Proceedings of the IEEE International Workshop on Computational Advances in Multisensor Adaptive Processing}, December 13-15, pages 129-132, 2005.


\bibitem{Qi2005}L.~Qi. Eigenvalues of a real supersymmetric tensor, {\it J. Symbolic Comput.,} 40 (2005) 1302-1324.

\bibitem{Shao2013a}J.~Shao, A general product of tensors with applications, {\it Linear Algebra Appl.,} 439 (2013) 2350-2366.
\bibitem{Shao2013}J.~Shao, H.~Shan, L.~Zhang, On some properties of the determinants of tensors, {\it Linear Algebra Appl.,} 439 (2013) 3057-3069.
\bibitem{Shao2015}J.~Shao, H.~Shan, B.~Wu, Some spectral properties and characterizations of connected odd-bipartite uniform hypergraphs, {\it Linear Multilinear Algebra,} 63 (2015) 2359-2372.
\bibitem{Bu2014}J.~Zhou, L. Sun, W.~Wang, C.~Bu, Some spectral properties of uniform hypergraphs, {\it Electron. J. Combin.,} 21(4) (2014)$\sharp$P4.24.
\bibitem{Yang2010}Y. Yang, Q. Yang, Further results for Perron-Frobenius theorem for nonnegative tensors, {\it SIAM J. Matrix Anal. Appl.,} 31
(2010) 2517-2530.
\bibitem{Yang2011}Q. Yang, Y. Yang, Further results for Perron-Frobenius theorem for nonnegative tensors II, {\it SIAM J. Matrix Anal. Appl.,} 32
(2011) 1236-1250.
\bibitem{Yuan2016}X.~Yuan, L.~Qi, J.~Shao, The proof of a conjecture on largest Laplacian and signless Laplacian H-eigenvalues of uniform hypergraphs, {\it Linear Algebra Appl.,} 490 (2016) 18-30.
\bibitem{Yue2016}J.~Yue, L.~Zhang, M.~Lu, Largest adjacency, signless Laplacian, and Laplacian H-eigenvalues of loose paths, {\it Front. Math. China,} 11 (2016) 623-645.
\end{thebibliography}
\end{document}